\documentclass[12pt,bezier]{article}

\usepackage{mathrsfs}
\usepackage{graphics}
\usepackage{latexsym}
\usepackage{amsmath}
\usepackage{amsfonts}
\usepackage{amssymb}
\usepackage{epsf}
\usepackage{cases}
\usepackage{mathtools}
\usepackage{tabularx}
\usepackage{latexsym,bm}
\usepackage{amsthm}
\usepackage{cite}

\newtheorem{lem}{Lemma}[section]
\newtheorem{thm}[lem]{Theorem}
\newtheorem{cor}[lem]{Corollary}

\newtheorem{pro}[lem]{Proposition}

\begin{document}

\title{On the spanning connectivity of tournaments}%\footnote{The research is supported by CSC.}
\author{ Bo zhang, \quad Weihua Yang\footnote{Corresponding author. E-mail: ywh222@163.com, yangweihua@tyut.edu.cn.},\quad Shurong  Zhang\\
\small Department of Mathematics, Taiyuan University of
Technology, Taiyuan 030024,
China\\
}
\date{}
\maketitle

{\small{\bf Abstract.} \quad  Let $D$ be a digraph. A $k$-container of $D$ between $u$ and $v$, $C(u,v)$, is a set of $k$ internally disjoint paths between $u$ and $v$. A $k$-container $C(u,v)$ of $D$ is a strong (resp. weak) $k^{*}$-container if there is a set of $k$ internally disjoint paths with the same direction (resp. with different directions allowed) between $u$ and $v$ and it contains all vertices of $D$.  A digraph $D$ is $k^{*}$-strongly (resp. $k^{*}$-weakly) connected if there exists a strong (resp. weak) $k^{*}$-container between any two distinct vertices. We define the strong (resp. weak) spanning connectivity of a digraph $D$, $\kappa_{s}^{*}(D)$ (resp. $\kappa_{w}^{*}(D)$ ), to be the largest integer $k$ such that $D$ is $\omega^{*}$-strongly (resp. $\omega^{*}$-weakly) connected for all $1\leq \omega\leq k$ if $D$ is a $1^{*}$-strongly (resp. $1^{*}$-weakly) connected. In this paper, we show that a tournament with $n$ vertices and irregularity $i(T)\leq k$, if $n\geq6t+5k$ $(t\geq2)$, then $\kappa_{s}^{*}(T)\geq t$ and $\kappa_{w}^{*}(T)\geq t+1$ if $n\geq6t+5k-3$ $(t\geq2)$.

\section{Introduction}

For terminology not explicitly introduced here, we refer to \cite{bang,four}. A digraph $D$ consists of a set $V(D)$ of vertices and a set $A(D)$ of order pairs $xy$ of distinct vertices called arcs, if $xy$ is an arc of $D$, we say that $x$ dominates $y$. A tournament $T$ is a digraph such that each pair of vertices is joined by precisely one arc. A Hamiltonian path (resp. cycle) of a tournament is a path (resp. cycle) including all vertices of $T$. Bang-Jensen, Gutin and Huang \cite{eleven} obtained some characterization for the existence of a Hamiltonian $(x,y)$-path in an extended tournament. If $P=x_{1}x_{2}...x_{m}$ is a path in a tournament $T$ and $x$ is a vertex not in the path such that $x$ is dominated by some $x_{i}$ and dominates some $x_{j}$ with $i<j$, then $T$ contains an $(x_{1},x_{m})$-path $P^{'}=x_{1}x_{2}...x_{k}xx_{k+1}...x_{m}$ where $i\leq k\leq j-1$. We say that $P^{'}$ is an $augmentation$ of $P$. By Menger's theorem \cite{two}, a tournament $T$  is $k$-strongly connected (or called $k$-strong) if and only if for each ordered pair $x$,$y$ of vertices, $T$ contains $k$ internally disjoint paths from $x$ to $y$.
When $T$ is not strong, and let $T_{1}$, $T_{2}$, ..., $T_{k}$ be the components. Without loss of generality, we may assume that whenever $i<j$, each vertex of $T_{i}$ dominates each vertex of $T_{j}$. We refer to $T_{i}$ as the $i'$th component of $T$, and to $T_{1}$ and $T_{k}$ as the initial and terminal components respectively. When $T$ is strong, $T$ has a Hamiltonian cycle. If $uv\in A(T)$ and $P$ is a $(u,v)$-path of length $k$, then $P$ is called a $k$-$bypass$ of $uv$. Alspach et al. in \cite{one} investigates bypasses in asymmetric digraphs. Zhang and Wu \cite{six} investigates conditions for arc-3-cyclicity in tournaments. Guo et al. in \cite{seven} investigates bypasses in tournaments.

We denote the in-degree and out-degree of vertex $x$ by $d^{-}(x)$ and $d^{+}(x)$ respectively. The $irregularity$ $i(T)$ of a tournament $T$ is the maximum $|d^{+}(x)-d^{-}(x)|$ over all vertices $x$ of $T$. Clearly, a tournament is regular if $i(T)=0$. If $i(T)\not=0$, the $T$ contains a vertex of in-degree at least $\frac{1}{2}(n-1)$ and a vertex of out-degree at least $\frac{1}{2}(n-1)$. Every vertex of $T$ has out-degree at least $\frac{1}{2}(n-1-i(T))$.

In an undirected graph $G$, a $k$-container of $G$ between $u$ and $v$, $C(u,v)$, is a set of $k$ internally disjoint paths between $u$ and $v$. The concept of container is proposed by Hsu in \cite{eight} to evaluate the performance of communication of an interconnection networks. A $k$-container $C(u,v)$ of $G$ is a $k^{*}$-container if it contains all vertices of $G$. A graph $G$ is $k^{*}$-connected if there exists a $k^{*}$-container between any two distinct vertices of $G$. The study of $k^{*}$-connected graph is motivated by the globally $3^{*}$-connected graphs proposed by Albert et al. in \cite{nine}. Lin et al. in \cite{ten} proved that the pancake graph $P_{n}$ is $w^{*}$-connected for any $w$ with $1\leq w\leq n-1$ if and only if $n\neq 3$. Later, they apply the concept to discuss the spanning connectivity of graphs in \cite{three} and discuss the spanning fan-connectivity of graphs in \cite{five}.

In this paper, we generalize the concepts above to digraphs and we consider the spanning connectivity of tournaments. Let $D$ be a digraph. A $k$-container of $D$ between $u$ and $v$, $C(u,v)$, is a set of $k$ internally disjoint paths between $u$ and $v$. A $k$-container $C(u,v)$ of $D$ is a strong (resp. weak) $k^{*}$-container if there is a set of $k$ internally disjoint paths with the same direction (resp. with different directions allowed) between $u$ and $v$ i.e. they have either $k$ $(u,v)$-paths or $k$ $(v,u)$-paths and it contains all vertices of $D$.  A digraph $D$ is $k^{*}$-strongly (resp. $k^{*}$-weakly) connected if there exists a strong (resp. weak) $k^{*}$-container between any two distinct vertices. We define the strong (resp. weak) spanning connectivity of a digraph $D$, $\kappa_{s}^{*}(D)$ (resp. $\kappa_{w}^{*}(D)$ ), to be the largest integer $k$ such that $D$ is $\omega^{*}$-strongly (resp. $\omega^{*}$-weakly) connected for all $1\leq \omega\leq k$ if $D$ is a $1^{*}$-strongly (resp. $1^{*}$-weakly) connected. We prove that for $k\geq 1$, a $2k$-strong tournament which has at least one path of length 2 between any vertices $x$ and $y$ is $(k+2)^{*}$-weakly connected, and that for $k\geq 0$, a $(2k+1)$-strong tournament is $(k+2)^{*}$-weakly connected. Further, for $k\geq 2$, we prove that a $2k$-strong tournament is $k^{*}$-strongly connected, and that a $(2k+1)$-strong tournament which has at least one 2-bypass for each arc is $(k+1)^{*}$-strongly connected. Finally, we also prove that a tournament with $n$ vertices and the irregularity $i(T)\leq k$, if $n\geq6t+5k$ $(t\geq2)$, then $\kappa_{s}^{*}(T)\geq t$ and that a tournament with $n$ vertices and the irregularity $i(T)\leq k$, if $n\geq6t+5k-3$ $(t\geq2)$, then $\kappa_{w}^{*}(T)\geq t+1$.

\section{$k^{*}$-weakly connected tournaments}

Several results proved in \cite{four} are useful in this paper, as follows.

\begin{thm}[\cite{four}]\label{1} Let $T$ be a tournament and $x$, $y$ distinct vertices of $T$. Then $T$ has a Hamiltonian path from $x$ to $y$ or from $y$ to $x$ unless one condition of $(i)-(iv)$ below is satisfied, in which case $T$ has no Hamiltonian path connecting $x$ and $y$.

$(i)$ $T$ is not strong and either the initial or the terminal component of $T$ (or both) contains none of $x$, $y$.

$(ii)$ $T$ is strong, $T-x$ is not strong, and $y$ belongs to neither the initial nor the terminal component of $T-x$.

$(iii)$ $T$ is strong, $T-y$ is not strong, and $x$ belongs to neither the initial nor the terminal component of $T-y$.

$(iv)$ $T$ is isomorphic to $T_{6}^{s}$ or $\overline{T}_{6}^{s}$ and $x$ and $y$ are as shown in Fig.2 (or interchanged on this figure).
\end{thm}

\begin{cor}[\cite{four}]\label{2} For any three vertices of a strong tournament there is a Hamiltonian path connecting two of them.
\end{cor}

\begin{cor}[\cite{four}] A 2-strong tournament is $1^{*}$-weakly connected unless it is isomorphic to $T_{6}^{s}$ or $\overline{T}_{6}^{s}$.
\end{cor}

 Thomassen also prove that every arc of a 3-strong tournament is contained in a Hamiltonian cycle i.e. for any two vertices $x$, $y$ of a 3-strong tournament, there is a Hamiltonian path between $x$ and $y$, and that every arc of a 4-strong tournament has a Hamiltonian bypass i.e. for any two vertices $x$, $y$ of a 4-strong tournament, there is a Hamiltonian path from $x$ to $y$ and from $y$ to $x$. According to these conclusions, the following is easy to obtain.

\begin{cor}[\cite{four}]\label{3} A 3-strong tournament is $1^{*}$-weakly connected.
\end{cor}

\begin{cor}[\cite{four}]\label{4} A 4-strong tournament is $1^{*}$-strongly connected and $2^{*}$-strongly connected.
\end{cor}

The following is also easy to obtain.

\begin{pro}\label{5} A strong tournament is $2^{*}$-weakly connected.
\end{pro}

\begin{proof} Let $T$ be a strong tournament, then $T$ has a Hamiltonian cycle. So for any two vertices $x$ and $y$ of $T$, there is a weak $2^{*}$-container between $x$ and $y$. Therefore, $T$ is $2^{*}$-weakly connected.
\end{proof}

\begin{thm}\label{10} For all $k\geq 0$, a $(2k+1)$-strong tournament is $(k+2)^{*}$-weakly connected.
\end{thm}

\begin{proof} Let $T$ be a $(2k+1)$-strong tournament and $x$, $y$ be any two vertices of $T$. The vertices of $T-\{x, y\}$ can be partitioned into four subsets $A$, $B$, $C$, $D$ such that both $x$ and $y$ dominate each vertex of $A$ and are dominated by each vertex of $B$, each vertex of $C$ dominates $y$ and is dominated by $x$, and each vertex of $D$ dominates $x$ and is dominated by $y$. Now, we consider the proposition for $k\geq 0$ as follows.

Case 1. $C\cup D=\emptyset$. $T$ has no path of length 2 between $x$ and $y$. Since $T$ is $(2k+1)$-strong, there are at least $2k+1$ arcs from $A$ to $B$. Thus, $T$ has at least $2k+1$ internally disjoint paths of length 3 between $x$ and $y$. We choose $k$ internally disjoint paths of length 3 among them and delate those $2k$ intermediate vertices of paths of length 3, the resulting tournament is strong. By Proposition \ref{5}, there is a weak $2^{*}$-container between $x$ and $y$. Combining the $k$ internally disjoint paths of length 3 between $x$ and $y$, then $T$ has a weak $(k+2)^{*}$-container between $x$ and $y$.

Case 2. $1\leq |C\cup D|\leq k-1$. We assume $|C\cup D|=i$. Since $T$ is $(2k+1)$-strong, as an argument similar to that used above in the case 2 of Proposition \ref{11}, there are at least $k-i$ arcs from $A$ to $B$. Then $T$ has at least $k-i$ paths of length 3 between $x$ and $y$. Delate those $2(k-i)$ intermediate vertices of paths of length 3 and those $i$ vertices of $C\cup D$, then $T$ is $(i+1)$-strong. So the resulting tournament is at least strong and has a weak $2^{*}$-container between $x$ and $y$. Combining the $k-i$ internally disjoint paths of length 3 and the $i$ internally disjoint paths of length 2 between $x$ and $y$, then $T$ has a weak $(k+2)^{*}$-container between $x$ and $y$.

Case 3. $|C\cup D|\geq k$. $T$ has at least $k$ paths of length 2 between $x$ and $y$. Delate the $k$ intermediate vertices of paths of length 2, the resulting tournament is strong, by Proposition \ref{5}, there is a weak $2^{*}$-container between $x$ and $y$. Combining the $k$ internally disjoint paths of length 2 between $x$ and $y$, then $T$ has a weak $(k+2)^{*}$-container between $x$ and $y$.

Hence, $T$ is $(k+2)^{*}$-weakly connected. The proof is completed.
\end{proof}

By an argument similar to that of Theorem~\ref{10}, we obtain the following proposition in which the connectivity can be weaken slightly.

\begin{pro}\label{11} If a $2k$-strong tournament $T$ has at least one path of length 2 between any two vertices $x$ and $y$, then $T$ is $(k+2)^{*}$-weakly connected.
\end{pro}

\begin{proof} Let $x$, $y$ be any two vertices of $T$. The vertices of $T-\{x, y\}$ can be partitioned into four subsets $A$, $B$, $C$, $D$ such that both $x$ and $y$ dominate each vertex of $A$ and are dominated by each vertex of $B$, each vertex of $C$ dominates $y$ and is dominated by $x$, and each vertex of $D$ dominates $x$ and is dominated by $y$. We assume that there are $i$ paths of length 2 between $x$ and $y$ i.e. $|C\cup D|=i$ and $i\geq 1$.

Case 1. $|C\cup D|=1$. Since $T$ is $2k$-strong,  $T$ has at least $2k-1$ paths of length 3 between $x$ and $y$. We choose $k-1$ paths of length 3 between $x$ and $y$, delate those $2(k-1)$ intermediate vertices of paths of length 3 and the vertex in $C$ or $D$. The resulting tournament is strong, by Proposition \ref{5}, there is a weak $2^{*}$-container between $x$ and $y$. Combining the $k-1$ internally disjoint paths of length 3 and the path of length 2 between $x$ and $y$, then $T$ has a weak $(k+2)^{*}$-container between $x$ and $y$.

Case 2. $2\leq |C\cup D|\leq k-1$. If $|C|=|D|=\frac{i}{2}$, since $T$ is $2k$-strong, there are at least $2k-\frac{i}{2}$ arcs from $A$ to $B$. Then $T$ has at least $2k-\frac{i}{2}$ paths of length 3 between $x$ and $y$. So we choose $k-i$ paths of length 3 between $x$ and $y$. Delate those $2(k-i)$ intermediate vertices of paths of length 3 and those $i$ intermediate vertices of paths of length 2. The resulting tournament is strong, by Proposition \ref{5}, there is a weak $2^{*}$-container between $x$ and $y$. Combining the $k-i$ internally disjoint paths of length 3 and the $i$ internally disjoint paths of length 2 between $x$ and $y$, then $T$ has a weak $(k+2)^{*}$-container between $x$ and $y$. If $|C|<|D|$ or $|D|<|C|$, we assume $min\{|C|,|D|\}=m$ and $m<\frac{i}{2}$. Since $T$ is $2k$-strong, there are at least $2k-m-1$ arcs from $A$ to $B$. Then $T$ has at least $2k-m-1$ paths of length 3 between $x$ and $y$. So we choose $k-i$ paths of length 3 between $x$ and $y$. By an argument similar to that used above in this case, one can see that $T$ has a weak $(k+2)^{*}$-container between $x$ and $y$.

Case 3. $|C\cup D|\geq k$. $T$ has at least $k$ paths of length 2 between $x$ and $y$. Delate the $k$ intermediate vertices of paths of length 2, the resulting tournament is strong, by Proposition \ref{5}, there is a weak $2^{*}$-container between $x$ and $y$. Combining the $k$ internally disjoint paths of length 2 between $x$ and $y$, then $T$ has a weak $(k+2)^{*}$-container between $x$ and $y$.

Hence, $T$ is $(k+2)^{*}$-weakly connected. The proof is completed.
\end{proof}

\section{$k^{*}$-strongly connected tournaments}

\begin{thm}
For all $k\geq 2$, a $2k$-strong tournament is $k^{*}$-strongly connected.
\end{thm}

\begin{proof} Let $T$ be a $2k$-strong tournament and $x$, $y$ be any two vertices of $T$. Without loss of generality, we assume $xy\in A(T)$. The vertices of $T-\{x, y\}$ can be partitioned into four subsets $A$, $B$, $C$, $D$ such that both $x$ and $y$ dominate each vertex of $A$ and are dominated by each vertex of $B$, each vertex of $C$ dominates $y$ and is dominated by $x$, and each vertex of $D$ dominates $x$ and is dominated by $y$. Now, we consider three cases as follows.

Case 1. $C=\emptyset$. Since $T$ is $2k$-strong, there are at least $2k-1$ arcs from $A$ to $B$. Thus, $T$ has $2k-1$ internally disjoint $(x,y)$-paths of length 3. We choose $k-2$ internally disjoint $(x,y)$-paths among them and delate the $2(k-2)$ intermediate vertices of the $(x,y)$-paths. The resulting tournament is 4-strong, by Corollary \ref{4}, there is a strong $2^{*}$-container from $x$ to $y$. Combining the $k-2$ internally disjoint $(x,y)$-paths of length 3, $T$ has a strong $k^{*}$-container from $x$ to $y$.

Case 2. $1\leq |C|\leq k-3$. Since $T$ is $2k$-strong, there are at least $2k-1-|C|$ arcs from $A$ to $B$. Thus, $T$ has $2k-1-|C|$ internally disjoint $(x,y)$-paths of length 3. We choose $k-2-|C|$ internally disjoint $(x,y)$-paths among them, then delate the $2(k-2-|C|)$ intermediate vertices of $(x,y)$-paths of length 3 and the vertices of $C$. The resulting tournament is $(4+|C|)$-strong, by Corollary \ref{4}, there is a strong $2^{*}$-container from $x$ to $y$. Combining the $k-2-|C|$ internally disjoint $(x,y)$-paths of length 3 and the $|C|$ internally disjoint $(x,y)$-paths of length 2, $T$ has a strong $k^{*}$-container from $x$ to $y$.

Case 3. $|C|\geq k-2$. $T$ has at least $k-2$ internally disjoint $(x,y)$-paths of length 2. Delate the $k-2$ intermediate vertices of $(x,y)$-paths of length 2, the resulting tournament is $(k+2)$-strong and it is at least 4-strong, by Corollary \ref{4}, there is a strong $2^{*}$-container from $x$ to $y$. Combining the $k-2$ internally disjoint $(x,y)$-paths of length 2, $T$ has a strong $k^{*}$-container from $x$ to $y$.

Hence, $T$ is $k^{*}$-strongly connected. The proof is completed.
\end{proof}

By an argument similar to that of the theorem above, we obtain the following theorem in which the connectivity can be weaken slightly and we omit the detailed proof.
\begin{thm} For all $k\geq 2$, let $T$ be a $2k+1$-strong tournament. If $T$ has at least one 2-bypass for each arc of $T$, then $T$ is $(k+1)^{*}$-strongly connected.
\end{thm}

\section{Spanning connectivity of tournaments}

\begin{lem}[\cite{four}]\label{9} Let $T$ be a tournament with $n$ vertices and the irregularity $i(T)\leq k$. Then the connectivity of $T$ is at least $\frac{1}{3}(n-2k)$.
\end{lem}

\begin{thm} Let $T$ be a tournament with $n$ vertices and $i(T)\leq k$. If $n\geq6t+5k$ $(t\geq 2)$ , then $\kappa_{s}^{*}(T)\geq t$.
\end{thm}

\begin{proof} By Lemma \ref{9}, $T$ is $2t$-strong. As $t\geq2$, then $T$ is 4-strong. By Corollary \ref{4}, $T$ is $i^{*}$-strongly connected for $i\in\{1,2\}$. Let $x$, $y$ be any two vertices of $T$. Without loss of generality, we assume $x$ dominates $y$. Let $A$ (resp. $B$) be the set of vertices in $T-\{x,y\}$ dominated by $x$ (resp. dominating $y$). Now we consider the following three cases.

Case 1. $A\cap B=\emptyset$.

We shall show that there are at least $t-2$ vertices of $A$ dominate at least $t-2$ vertices of $B$. Since $|A|\geq\frac{1}{2}(n-3-k)$, $|B|\geq\frac{1}{2}(n-3-k)$ , we have $|C|=|V(T)\backslash(A\cup B\cup\{x,y\})|\leq n-2-\frac{1}{2}(n-3-k)-|A|$. Without loss of generality, we may suppose each vertex $u\not\in\{u_{1}, u_{2},...,u_{t-2}\}$ satisfying $d^{-}(u)\leq d^{-}(u_{i})$ for $i\in\{1,2,...,t-2\}$ in the sub-tournament $T[A]$. Recall that every tournament $T^{'}$ has a vertex of in-degree at least $\frac{1}{2}(|T^{'}|-1)$. This means that the in-degree of $u_{i}$ in $T[A]$ satisfies $d^{-}_{A}(u_{_i})\geq\frac{1}{2}(|A|-t+2)$ for $i\in\{1,2,...,t-2\}$. Otherwise, there would be a vertex in $A\backslash\{u_{1},u_{2},...,u_{t-2}\}+u_{i}$ such that the in-degree of it larger than $u_{i}$, contradicting the choice of $u_{i}$. So the out-degree of $u_{i}$ in $T[A]$ satisfies $d^{+}_{A}(u_{i})=|A|-1-d^{-}_{A}(u_{i})\leq \frac{1}{2}|A|+\frac{1}{2}(t-4)$. Therefore, the number of vertices in $B$ dominated by $u_{i}$ is at least $\frac{1}{2}(n-1-k)-d^{+}_{A}(u_{i})-d^{+}_{C}(u_{i})\geq t+\frac{5}{4}$, where $d^{+}_{C}(u_{i})$ denotes the number of vertices in $C$ dominated by $u_{i}$. It can be seen that every vertex $u_{i}$ dominates at least $t-2$ vertices of $B$. Thus, we obtain $t-2$ internally disjoint paths $xu_{i}v_{j}y$ of length 3, where $v_{j}\in B$ and $j\in\{1,2,...,t-2\}$. Let $\{P_{1}, P_{2},...,P_{m}\}$ $(1\leq m\leq t-2)$ denote the $m$ internally disjoint $(x,y)$-path of length 3. Deleting the $2m$ vertices from $T$, the resulting tournament is  4-strong. By Corollary \ref{4}, the resulting tournament is $2^{*}$-strongly connected. As $x$ dominates $y$, there is a strong $2^{*}$-container from $x$ to $y$. Let $\{P_{m+1}, P_{m+2}\}$ denote the two internally disjoint $(x,y)$-paths. Then $\{P_{1},P_{2},...,P_{m+2}\}$ forms a strong $(m+2)^{*}$-container from $x$ to $y$. Hence, $T$ has a strong $i^{*}$-container from $x$ to $y$ for all $i\in\{3,4,...,t\}$.

Case 2. $|A\cap B|\geq t-2$.

We choose $m$ $(1\leq m\leq t-2)$ vertices from $A\cap B$ and the $m$ vertices imply a set of $m$ internally disjoint $(x,y)$-paths of length 2, denoted by $\{P_{1},P_{2},...,P_{m}\}$. Deleting the $m$ vertices, the resulting tournament is $(2t-m)$-strong. Clearly, the resulting tournament is 4-strong. By Corollary \ref{4}, it is $2^{*}$-strongly connected. As $x$ dominates $y$, there is a strong $2^{*}$-container from $x$ to $y$. Let $\{P_{m+1}, P_{m+2}\}$ denote the two internally disjoint $(x,y)$-paths. Thus, $\{P_{1},P_{2},...,P_{m+2}\}$ forms a strong $(m+2)^{*}$-container from $x$ to $y$. Hence, $T$ has a strong $i^{*}$-container from $x$ to $y$ for all $i\in\{3,4,...,t\}$.

Case 3. $1\leq |A\cap B|\leq t-3$.

Suppose $A\cap B=S$ and $|S|=s$. Since $|A|\geq\frac{1}{2}(n-3-k)$, $|B|\geq\frac{1}{2}(n-3-k)$, we have $|C|=|V(T)\backslash(A\cup B\cup\{x,y\})|\leq n-2-\frac{1}{2}(n-3-k)-|A|+s$. We assume each vertex $u\not\in\{u_{1}, u_{2},...,u_{t-2-s}\}$ satisfying $d^{-}(u)\leq d^{-}(u_{i})$ for all $i\in \{1,2,...,t-2-s\}$ in the sub-tournament $T[A\backslash B]$. The in-degree of $u_{i}$ in $T[A\backslash B]$ satisfies $d^{-}_{A\backslash B}(u_{_i})\geq\frac{1}{2}(|A\backslash B|-t+2+s)$ for each $i\in\{1,2,...,t-2-s\}$. Otherwise, there would be a vertex in $A\backslash (B\cup \{u_{1},u_{2},...,u_{t-2-s}\})+u_{i}$ such that the in-degree of it larger than $u_{i}$, contradicting the choice of $u_{i}$. So the out-degree of $u_{i}$ in $T[A\backslash B]$ satisfies $d^{+}_{A\backslash B}(u_{i})=|A\backslash B|-1-d^{-}_{A\backslash B}(u_{i})\leq \frac{1}{2}|A\backslash B|+\frac{1}{2}(t-4-s)$. Therefore, the number of vertices in $B$ dominated by $u_{i}$ is at least $\frac{1}{2}(n-1-k)-d^{+}_{A\backslash B}(u_{i})-d^{+}_{C}(u_{i})\geq t+\frac{5}{4}$, where $d^{+}_{C}(u_{i})$ denote the number of vertices in $C$ dominated by $u_{i}$. It can be seen that every vertex $u_{i}$ dominates at least $(t-2-s)$ vertices of $B\backslash A$.
Thus, we obtain $(t-2-s)$ internally disjoint paths $xu_{i}v_{j}y$ of length 3 $(v_{j}\in B, j\in\{1,2,...,t-2-s\})$ and $s$ internally disjoint $(x,y)$-paths of length 2. Let $\{P_{1}, P_{2},...,P_{s}\}$ $(1\leq s\leq t-3)$ denotes the $s$ internally disjoint $(x,y)$-paths of length 2 and $\{Q_{1},Q_{2},...,Q_{m}\}$ $(1\leq m\leq t-2-s)$ denotes the $m$ internally disjoint $(x,y)$-paths of length 3. Deleting the $2m+s$ vertices from $T$, the resulting tournament is $(4+s)$-strong. By Corollary \ref{4}, the resulting tournament is $2^{*}$-strongly connected. As $x$ dominates $y$, there is a strong $2^{*}$-container from $x$ to $y$. Let $\{R_{1}, R_{2}\}$ denote the two internally disjoint $(x,y)$-paths. Then $\{P_{1}, P_{2},..., P_{s}, Q_{1}, Q_{2},..., Q_{m}, R_{1}, R_{2}\}$ forms a strong $(s+m+2)^{*}$-container from $x$ to $y$. Hence, $T$ has a strong $i^{*}$-container from $x$ to $y$ for all $i\in\{3,4,...,t\}$.

Thus, $T$ is $i^{*}$-strongly connected for $i\in\{3,4,...,t\}$. As $T$ is $i^{*}$-strongly connected for $i\in\{1,2\}$., then $\kappa_{s}^{*}(T)\geq t$. The proof is completed.
\end{proof}

By an argument similar to that used above, it is not difficult to obtain the following. We omit the detailed proof. 

\begin{thm} Let $T$ be a tournament with $n$ vertices and $i(T)\leq k$. If $n\geq6t+5k-3$ $(t\geq2)$, then $\kappa_{w}^{*}(T)\geq t+1$.
\end{thm}

\section{Acknowledgements}

The second author would like to thank professor Lih-Hsing Hsu for posing the problems considered in the paper. The research is supported by NSFC (No.11671296,  61502330), SRF for ROCS, SEM and Fund Program for the Scientific Activities of Selected
Returned Overseas Professionals in Shanxi Province.

\end{document}